%% file: main.tex
\documentclass{article}
\author{Alexandre Goy}
\usepackage[utf8]{inputenc}

\title{Infinite Trees}
\date{\today}

\usepackage{amsmath}
\usepackage{amsthm}
\usepackage{amssymb}
\usepackage{stmaryrd}

\usepackage{tikzit}
\input{trees.tikzstyles}

\theoremstyle{definition}
\newtheorem{theorem}{Theorem}
\newtheorem{lemma}[theorem]{Lemma}
\newtheorem{proposition}[theorem]{Proposition}
\newtheorem{corollary}[theorem]{Corollary}

\theoremstyle{definition}
\newtheorem{definition}[theorem]{Definition}
\newtheorem{example}[theorem]{Example}

\theoremstyle{remark}
\newtheorem*{remark}{Remark}

\newcommand{\CB}{\mathsf{CB}}

\newcommand{\N}{\mathbb{N}}

\newcommand{\Pref}{\mathsf{Pref}}

\begin{document}

\maketitle

\abstract{A few notes about infinite trees in a descriptive set-theoretic setting.}

\tableofcontents

\section{Basics on Languages}

This section introduces notation for manipulating words and languages.\\

Let $\omega = \{0,1,\ldots\}$ be the first infinite ordinal, identified with $\N$. Let $\Sigma$ be a set called the \emph{alphabet}. Let $\Sigma^* = \bigcup_{n \in \omega} \Sigma^n$ be the set of finite words on $\Sigma$, $\Sigma^\omega$ be the set of infinite words on $\Sigma$, and $\Sigma^\infty = \Sigma^* + \Sigma^\omega = \bigcup_{\alpha \in \omega + 1} \Sigma^\alpha$ be the set of all words. Letters in a word are numbered starting from $1$. The length of a word $w$ is denoted by $|w|$, defining a length map $|-| : \Sigma^\infty \to \omega + 1$. Words will be typically written
\begin{align*}
     & w = w_1 w_2 \ldots w_n & \text{if } w \in \Sigma^n      & \text{ i.e. } |w| = n      \\
     & w = w_1 w_2 \ldots     & \text{if } w \in \Sigma^\omega & \text{ i.e. } |w| = \omega
\end{align*}
There is a unique word of length $0$ called the empty word and denoted by $\varepsilon$. Words of length $1$ and letters are identified, i.e., $\Sigma^1 \cong \Sigma$.

\paragraph{Concatenation.}
Concatenation of words has type $\Sigma^\infty \times \Sigma^\infty \to \Sigma^\infty$ and is denoted by juxtaposition. With the convention that left-concatenating an infinite word overwrites whatever is on the right, concatenation is associative with neutral element $\varepsilon$. For instance, $a^\omega b^\omega = a^\omega$. Fixing the first argument $u \in \Sigma^\infty$ yields a map $u : \Sigma^{\infty} \to \Sigma^\infty$. Applying covariant and contravariant powerset to this map generates
\begin{align*}
    u(\cdot) \colon & P(\Sigma^\infty) \to P(\Sigma^\infty) & u^{-1}(\cdot) \colon & P(\Sigma^\infty) \to P(\Sigma^\infty)           \\
                    & L \mapsto \{uv \mid v \in L\}         &                      & L \mapsto \{v \in \Sigma^\infty \mid uv \in L\}
\end{align*}
Note that these maps are degenerate when $u$ is an infinite word, as in that case $uL = \{u\}$ if $L \neq \emptyset$, $\emptyset$ else and $u^{-1}L = \Sigma^\infty$ if $u \in L$, $\emptyset$ else.
\begin{proposition}
    For all $u, v \in \Sigma^\infty$ and $L \in P(\Sigma^\infty)$
    \begin{enumerate}
        \item $u(vL) = (uv)L$
        \item $u^{-1}(v^{-1}L) = (vu)^{-1}L$
        \item $uu^{-1}L = L \cap u\Sigma^\infty$
        \item $u^{-1}uL = L$ if $u \in \Sigma^*$, $\emptyset$ if $L = \emptyset$, $\Sigma^\infty$ else
    \end{enumerate}
\end{proposition}
\begin{proof} Sketch:
    \begin{enumerate}
        \item This is functoriality of the covariant powerset.
        \item This is functoriality of the contravariant powerset.
        \item We have $uu^{-1}L = \{uv \mid uv \in L\} = L \cap u\Sigma^\infty$.
        \item We have $u^{-1}uL = \{v \in \Sigma^\infty \mid uv \in uL\}$. If $u \in \Sigma^*$, $uv \in uL \iff v \in L$ whence the result, else and if $L \neq \emptyset$, $uv \in uL \iff u \in \{u\}$ whence the result.
    \end{enumerate}
\end{proof}

Concatenation of languages $P(\Sigma^\infty) \times P(\Sigma^\infty) \to P(\Sigma^\infty)$ is also denoted by juxtaposition.

\paragraph{Restriction and Co-Restriction.}

The restriction of a word $w  \in \Sigma^\infty$ to a finite length $n \in \omega$ is defined by $w_{\mid n} = w_1 w_2 \ldots w_n$ if $n \leq |w|$, $w_{\mid n} = w$ else. For instance, $w_{|0} = \varepsilon$ is the empty word. The co-restriction of a word $w \in \Sigma^\infty$ to a finite length $n \in \omega$ is the unique word $w^{\mid n} \in \Sigma^\infty$ such that $w = w_{\mid n} w^{\mid n}$. For instance, $w^{\mid 0} = w$ and $w^{\mid n} = \varepsilon$ for any $n \geq |w|$. Note that $w^{\mid n} = w_{n+1} w_{n+2} \ldots$ for infinite words. Furthermore, for every $w \in \Sigma^\infty$ and $p,q \in \omega$ the equality
$w_{\mid p+q} = w_{\mid p} (w^{\mid p})_{\mid q}$ holds.

\paragraph{Prefixing.}

The map $\Pref : \Sigma^\infty \to P(\Sigma^*)$ assigns to a word the set of its finite prefixes, i.e., $\Pref(w) = \{w_{\mid n} \mid n \in \omega\}$. It is extended to $\Pref : P(\Sigma^\infty) \to P(\Sigma^*)$ by union. We have
\[ \Pref(L) = \{w_{\mid n} \mid w \in L, n \in \omega\} = \{u \in \Sigma^* \mid L \cap u\Sigma^\infty \neq \emptyset \} \]

\begin{proposition}
    For any $L, M \in P(\Sigma^\infty)$,
    \begin{enumerate}
        \item $L \cap \Sigma^* \subseteq \Pref(L)$
        \item $L \subseteq M \Rightarrow \Pref(L) \subseteq \Pref(M)$
        \item $\Pref(\Pref(L)) = \Pref(L)$
        \item $\Pref(LM) = \Pref(L) \cup L\Pref(M)$ if $M \neq \emptyset$
    \end{enumerate}
\end{proposition}
\begin{proof} The first three statements are obvious. Note that this makes the restriction $\Pref : P(\Sigma^*) \to P(\Sigma^*)$ a closure operator. For the last one:
    \begin{enumerate}
        \item[4.] A prefix of a word in $LM$ is either a prefix of a word in $L$, or a word in $L$ then a prefix of a word in $M$, and conversely because $M$ is non-empty.
    \end{enumerate}
\end{proof}

\begin{example}[The binary alphabet]
    The \emph{binary alphabet} is $\Sigma = \{a,b\}$ where $a$ is interpreted as 'going left' and $b$ as 'going right'. The map $(-)^\perp : \Sigma \to \Sigma$ reverses directions as $a \mapsto b$ and $b \mapsto a$. It is extended letterwise to a map $(-)^\perp : \Sigma^\infty \to \Sigma^\infty$, for instance $(abba)^\perp = baab$.
\end{example}

\section{Basics on Trees}

This section introduces trees as they are usually defined in descriptive set theory. These trees can have finite or infinite depth, and they can have finite or infinite width depending on the alphabet $\Sigma$.\\

A language $L \in P(\Sigma^\infty)$ is \emph{prefix-closed} if $\Pref(L) = L$.

\begin{definition}[Tree]
    A \emph{tree} is a prefix-closed subset of $\Sigma^*$. The set of trees over $\Sigma$ is denoted by $T(\Sigma)$.
\end{definition}

\begin{proposition}
    The image of $\Pref : P(\Sigma^\infty) \to P(\Sigma^*)$ is $T(\Sigma)$.
\end{proposition}
\begin{proof}
    As $\Pref(\Pref(L)) = \Pref(L)$, the image is included in $T(\Sigma)$. Conversely, let $L \subseteq \Sigma^*$ be prefix-closed, then $L = \Pref(L)$ is in the image of $\Pref$.
\end{proof}

\begin{definition}
    Given a tree $T \in T(\Sigma)$, its set of infinite branches is
    \[ [T] = \{w \in \Sigma^\omega \mid \forall n \in \omega, w_{\mid n} \in T\} = \{w \in \Sigma^\omega \mid \Pref(w) \subseteq T\} \]
\end{definition}

\begin{proposition}
    For a tree $T \in T(\Sigma)$, $[T] = \{w \in \Sigma^\omega \mid \Pref(w) \cap T \text{ is infinite}\}$.
\end{proposition}

\begin{proof}
    We must show that for a tree $T$ and a word $w \in \Sigma^\omega$, $\Pref(w) \subseteq T$ iff $\Pref(w) \cap T$ is infinite. The left-to-right direction is obvious as $\Pref(w)$ is infinite. For the other direction, assume $\Pref(w) \cap T$ is infinite and let $n \in \omega$, we must show $w_{\mid n} \in T$. As $\Pref(w) \cap T$ is infinite, there is $N \geq n$ such that $w_{\mid N}  \in T$. As $T$ is prefix-closed and $n \leq N$, $w_{\mid n} = (w_{\mid N})_{\mid n} \in T$, whence the result.
\end{proof}

\begin{proposition}
    On the binary alphabet, taking infinite branches commutes with reversing directions: $[T^\perp] = [T]^\perp$.
\end{proposition}
\begin{proof}
    Compute
    \begin{align*}
        [T^\perp] & = \{w \in \Sigma^\omega \mid \forall n \in \omega, w_{\mid n} \in T^\perp\}    \\
                  & = \{w \in \Sigma^\omega \mid \forall n \in \omega, w_{\mid n}^\perp \in T\}    \\
                  & = \{w^\perp \mid w \in \Sigma^\omega, \forall n \in \omega, w_{\mid n} \in T\} \\
                  & = [T]^\perp
    \end{align*}
\end{proof}

\begin{proposition}
    Taking infinite branches commutes with taking arbitrary intersections: $[\bigcap_{i \in I} T_i] = \bigcap_{i \in I}[T_i]$.
\end{proposition}
\begin{proof}
    For any $w \in \Sigma^\omega$, $\forall n . \forall i . w_{\mid n} \in T_i \iff \forall i. \forall n . w_{\mid n} \in T_i$.
\end{proof}

\paragraph{Concatenation and Infinite Branches.}

\begin{proposition}
    For any $T \in T(\Sigma)$, $u \in \Sigma^*$, $a \in \Sigma$,
    \begin{enumerate}
        \item $u[T] = [\Pref(u) \cup uT]$
        \item $a[T] = [\{\varepsilon\} \cup aT]$
        \item $u^{-1}[T] = [u^{-1}T]$
    \end{enumerate}
\end{proposition}
\begin{proof}
    Let us prove the statements.
    \begin{enumerate}
        \item First check that $\Pref(u) \cup uT$ is a tree. We can assume $T$ is non-empty. Then $\Pref(uT) = \Pref(u) \cup u\Pref(T) = \Pref(u) \cup uT$ because $T$ is a tree, hence $\Pref(u) \cup uT$ is a tree. Now let us prove the wanted equality by double inclusion. Let $uw \in u[T]$, where $w \in [T]$. Then prefixes of $uw$ are, depending of their length, either in $\Pref(u)$, or in $uT$, so that $uw \in [\Pref(u) \cup uT]$. Conversely, let $w \in [\Pref(u) \cup uT]$. Let $n = |u|$. As $w_{\mid n} \in \Pref(u) \cup uT$ and the only element of length $n$ of that set is $u$ itself, we have $w_{\mid n} = u$. Furthermore, for length reasons $w_{\mid n+m} \in uT$ for every $m \in \omega$, therefore $w^{\mid n} \in [T]$. Whence $w = w_{\mid n} w^{\mid n} = uw^{\mid n} \in u[T]$.
        \item This is a special case of the previous point.
        \item First check that $u^{-1}T$ is a tree. We have
              \begin{align*}
                  \Pref(u^{-1}T) & = \{w_{\mid n} \mid w \in u^{-1}T \text{ and } n \in \omega\} \\
                                 & = \{w_{\mid n} \mid uw \in T \text{ and } n \in \omega\}      \\
                                 & \subseteq \{w \mid uw \in \Pref(T)\}                          \\
                                 & = u^{-1}\Pref(T) = u^{-1}T
              \end{align*}
              Now let us prove the wanted equality by double inclusion. We have
              \begin{align*}
                   & u^{-1}[T] = \{w \mid uw \in [T]\} = \{w \mid \Pref(uw) \subseteq T\}                 \\
                   & [u^{-1}T] = \{w \mid \Pref(w) \subseteq u^{-1}T\} = \{w \mid u\Pref(w) \subseteq T\}
              \end{align*}
              As $\Pref(uw) = \Pref(u) \cup u\Pref(w)$ and $T$ is prefix-closed, $\Pref(uw) \subseteq T$ iff $u\Pref(w) \subseteq T$, whence the result.
    \end{enumerate}
\end{proof}

Note that, given $u \in \Sigma^*$, the map $u^{-1}(\cdot) : T(\Sigma) \to T(\Sigma)$ recovers the sub-tree rooted at the node $u$. In particular
\begin{itemize}
    \item The tree $u^{-1}T$ is the sub-tree rooted at $u$ extracted from $T$.
    \item The language (usually not tree) $uu^{-1}T$ is that sub-tree viewed in its place inside $T$.
    \item The tree $\Pref(u) \cup uu^{-1}T$ is the sub-tree of $T$ given by keeping only the finite branch $u$ and the sub-tree rooted at $u$.
\end{itemize}

\paragraph{Graphical Representations.}
Any tree $T$ has a graphical representation given by drawing the relation $\{(u,ua) \mid u \in \Sigma^*, a \in \Sigma, ua \in T\}$. Unless $T$ is empty, this graph has a root $\varepsilon$ and branches that follow the spelling of words in $T$.

\begin{example}[The hat]
    The hat is the binary tree $H = \Pref(\{a^\omega,b^\omega\}) = \{a^n \mid n \in \omega\} \cup \{b^n \mid n \in \omega\}$. It has exactly two infinite branches, a left one and a right one, as $[H] = \{a^\omega,b^\omega\}$.
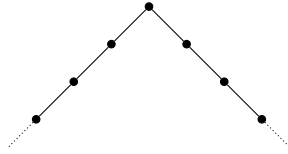
    \[
        \input{hat_tree.tikz}
    \]
\end{example}

\begin{example}[The $n$-branches] \label{ex:nbranches}
    We define inductively a family $\{B_n \mid n \in \omega\}$ of binary pruned trees.
    \begin{align*}
         & B_0 = \emptyset                                                        \\
         & B_{n+1} = \bigcup_{k \in \omega} a^k(\{\varepsilon\} \cup b B^\perp_n)
    \end{align*}
    Explicitly, we have $B_1 = \{a^k \mid k \in \omega\}$, $B_2 = \{a^k b^{k'} \mid k, k' \in \omega\}$ and more generally $B_n = \{a^{k_1} b^{k_2} \ldots (b^{\perp^n})^{k_n} \mid k_1, \ldots, k_n \in \omega\}$.
    These trees have the following infinite branches.
    \begin{align*}
         & [B_0] = \emptyset                                                \\
         & [B_1] = \{a^\omega\}                                             \\
         & [B_{n+1}] = \bigcup_{\alpha \in \omega + 1} a^\alpha [B_n]^\perp
    \end{align*}
    Explicitly, $[B_n] = \{a^{\alpha_1} b^{\alpha_2} \ldots (b^{\perp^{n-1}})^{\alpha_{n-1}} (b^{\perp^n})^{\omega} \mid \alpha_1, \ldots, \alpha_{n-1} \in \omega + 1\}$.
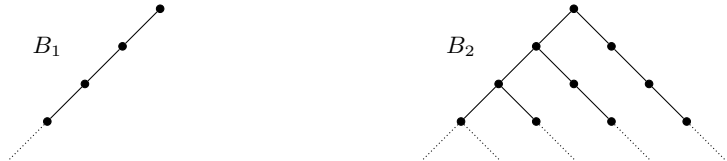
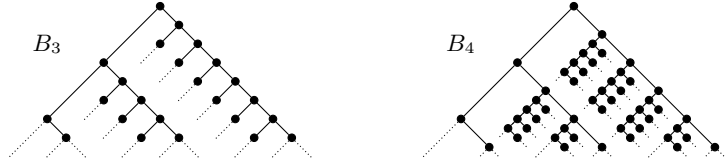
    \[ \input{branching1_tree.tikz} \]
    \[ \input{branching2_tree.tikz} \]
\end{example}

\begin{example}[The growing tree] \label{ex:growing}
    The \emph{growing tree} $G$ is defined as
    \begin{equation*}
        G = \bigcup_{k \in \omega} a^k (\{\varepsilon\} \cup b B^\perp_{k+1})
    \end{equation*}
    It consists in plugging all of the non-trivial $n$-branches into a lefmost infinite branch. Its infinite branches are
    \begin{equation*}
        [G] = \{a^k b^{\alpha_1} a^{\alpha_2} \ldots (a^{\perp^k})^{\alpha_k} (a^{\perp^{k+1}})^\omega \mid k \in \omega, \alpha_1, \ldots, \alpha_k \in \omega +1\}
    \end{equation*}
    \[ \input{growing_tree.tikz} \]
\end{example}

\subsection{Construction and Destruction of Trees}


There are two remarkable viewpoints on a tree: lying down and staring at what is around the root, or climbing along a branch.

\subsubsection{Around the Root}

Destructively speaking, a non-empty tree is the union of the root and its children subtrees.
\begin{proposition}
    For all $T \in T(\Sigma) \setminus \{\emptyset\}$, $T = \{\varepsilon\} \cup \biguplus_{a \in \Sigma} a a^{-1} T$.
\end{proposition}
\begin{proof}
    We have $T = \{\varepsilon\} \cup \biguplus_{a \in \Sigma} (T \cap a\Sigma^*) = \{\varepsilon\} \cup \biguplus_{a \in \Sigma} aa^{-1} T$.
\end{proof}

Constructively speaking, given a $\Sigma$-indexed family of trees $(T_a)_{a \in \Sigma}$, one can define
\[ T = \{\varepsilon\} \cup \biguplus_{a \in \Sigma} aT_a \]

This yields a tree whose infinite branches are obtained from those of the children subtrees.

\begin{proposition} \label{prop:root_construct}
    If $(T_a)_{a \in \Sigma} \in T(\Sigma)$ is a family of trees, then
    \begin{enumerate}
        \item $\{\varepsilon\} \cup \biguplus_{a \in \Sigma} aT_a \in T(\Sigma)$
        \item $[\{\varepsilon\} \cup \biguplus_{a \in \Sigma} aT_a] = \bigcup_{a \in \Sigma} a[T_a]$
    \end{enumerate}
\end{proposition}

\begin{proof}
    Let $T$ be the set $\{\varepsilon\} \cup \biguplus_{a \in \Sigma} aT_a$.
    \begin{enumerate}
        \item Let us show that $T$ is prefix-closed. Let $u \in T$, the case $u = \varepsilon$ is obvious so we can assume $u = av$ for some $a \in \Sigma$ and some $v \in T_a$. As $T_a$ is a tree, $\Pref(v) \subseteq T_a$ hence $a\Pref(v) \subseteq aT_a \subseteq T$. Therefore $\Pref(u) = \{\varepsilon\} \cup a\Pref(v) \subseteq T$, so that $T$ is prefix-closed.
        \item Let $w \in \Sigma^\omega$. Assume $w \in [T]$. We have for every $n \in \omega$, $w_{\mid n} \in \{\varepsilon\} \cup \biguplus_{a \in \Sigma} aT_a$. If $n$ is non-zero, the only possibility is that $w_{\mid n} \in w_1 T_{w_1}$. Therefore, every prefix of $w^{\mid 1}$ is in $T_{w_1}$, whence $w^{\mid 1} \in [T_{w_1}]$ so that $w \in w_1[T_{w_1}] \subseteq \bigcup_{a \in \Sigma} a[T_a]$. Conversely, let $w \in \bigcup_{a \in \Sigma} a[T_a]$, then there is $a \in \Sigma$ such that $w \in a[T_a]$. For every $n \in \omega$, either $n = 0$ and $w_{\mid 0} = \varepsilon \in T$, or $n > 0$ and $w_{\mid n} = au$ for some $u \in T_a$, therefore $w_{\mid n} \in aT_a \subseteq T$. This shows that $w \in [T]$.
    \end{enumerate}
\end{proof}

And as expected, we then have
\begin{proposition}
    If $T = \{\varepsilon\} \cup \biguplus_{a \in \Sigma} aT_a$ and $a \in \Sigma$, then $T_a = a^{-1}T$.
\end{proposition}
\begin{proof} As $a^{-1}a L = L$ and $a^{-1}b L = \emptyset$ for $b \neq a$, we have
    \[ a^{-1}T = a^{-1}\left(\{\varepsilon\} \cup \biguplus_{b \in \Sigma} bT_b \right) = a^{-1}\{\varepsilon\} \cup \biguplus_{b \in \Sigma} a^{-1}bT_b = T_a  \]
\end{proof}

In particular, for any non-empty tree $T \in T(\Sigma)$ we have
\begin{align*}
     & T = \{\varepsilon\} \uplus \biguplus_{a \in \Sigma} aa^{-1}T \\
     & [T] = \biguplus_{a \in \Sigma} aa^{-1}[T]
\end{align*}

\subsubsection{Along a Branch}

Destructively speaking, a tree with at least one infinite branch is the union of the subtrees plugged along that branch.

Given a infinite word $w \in \Sigma^\omega$, the set of \emph{off-words} of $w$ is defined as $w^\circ = \{w_{\mid n} a \mid n \in \omega, a \in \Sigma \setminus \{w_{n+1}\}\} \in P(\Sigma^*)$.

\begin{proposition} \label{prop:destruct_stream}
    For all $T \in T(\Sigma)$ and $w \in [T]$,
    \begin{align*} T & = \Pref(w) \cup \biguplus_{u \in w^\circ} uu^{-1}T                                                                                            \\
                 & = \biguplus_{n \in \omega} w_{\mid n} \left( \{\varepsilon\} \cup \biguplus_{a \in \Sigma \setminus \{w_{n+1}\}} a(w_{\mid n}a)^{-1} T\right)
    \end{align*}
\end{proposition}
\begin{proof}
    The set $\{\Pref(w)\} \cup \{u\Sigma^* \mid u \in w^\circ\}$ is clearly a partition of $\Sigma^*$. Therefore,
    \begin{align*}
        T = T \cap \Sigma^*
         & = (T \cap \Pref(w)) \cup \biguplus_{u \in w^\circ} (T \cap u\Sigma^*)                                                                         \\
         & = \Pref(w) \cup \biguplus_{u \in w^\circ} uu^{-1}T                                                                                            \\
         & = \biguplus_{n \in \omega} \left(\{w_{\mid n}\} \cup \biguplus_{u \in w^\circ \cap \Sigma^{n+1}} uu^{-1}T\right)                              \\
         & = \biguplus_{n \in \omega} \left(\{w_{\mid n}\} \cup \biguplus_{a \in \Sigma} w_{\mid n} a (w_{\mid n} a)^{-1}T\right)                        \\
         & = \biguplus_{n \in \omega} w_{\mid n} \left( \{\varepsilon\} \cup \biguplus_{a \in \Sigma \setminus \{w_{n+1}\}} a(w_{\mid n}a)^{-1} T\right)
    \end{align*}
\end{proof}

Constructively speaking, given a countable family of non-empty $\Sigma$-trees $(T_n)_{n\in \omega}$ and an infinite word $w \in \Sigma^\omega$, one can define

\begin{equation*}
    T = \bigcup_{n\in \omega} w_{\mid n} T_n
\end{equation*}

This yields a tree whose infinite branches are the main branch $w$ and the ones obtained from the subtrees $T_n$.

\begin{proposition} \label{prop:branch_construct}
    If $w \in \Sigma^\omega$ and $(T_n)_{n\in \omega} \in T(\Sigma)$ is a countable family of non-empty trees, then
    \begin{enumerate}
        \item $\bigcup_{n \in \omega} w_{\mid n} T_n \in T(\Sigma)$
        \item $\left[ \bigcup_{n \in \omega} w_{\mid n} T_n \right] = \{w\} \cup \bigcup_{n \in \omega} w_{\mid n} [T_n]$
    \end{enumerate}
\end{proposition}
\begin{proof}
    In the whole proof, let $U = \bigcup_{n \in \omega} w_{\mid n} T_n$.
    \begin{enumerate}
        \item Let $u \in w_{\mid n} T_n$ for some $n \in \omega$. Then $u = w_{\mid n} v$ for some $v \in T_n$. Prefixes of $u$ are either prefixes of $w$, which belong to $U$ because every $T_i$ contains $\varepsilon$, or of the form $w_{\mid n} v'$ for some $v'$ prefix of $v$. As $v \in T_n$ and $T_n$ is prefix-closed, $v' \in T_n$, hence $w_{\mid n} v' \in w_{\mid n} T_n$ is in $U$ as well.
        \item $\boxed{\supseteq}$ First consider $w$. For all $n \in \omega$, as $\varepsilon \in T_n$ we have $w_{\mid n} \in w_{\mid n} T_n$, therefore $w$ is an infinite branch of $U$. Now let $x \in [T_n]$ for some $n \in \omega$, let us show that $w_{\mid n} x$ is an infinite branch of $U$ as well. Every prefix of $w$ is in $U$, so consider $w_{\mid n} u$ where $u$ is a prefix of $x$. As $x \in [T_n]$ we have $u \in T_n$ therefore $w_{\mid n} u \in U$, which completes the proof.
        \item[] $\boxed{\subseteq}$ Let $b \in \left[ \bigcup_{n \in \omega} w_{\mid n} T_n \right] \setminus \{w\}$. Let $n_0 \in \omega$ be such that $b_{\mid n_0} \neq w_{\mid n_0}$ and minimal with that property. We know that for all $m \in \omega$, $b_{\mid m} \in \bigcup_{n \in \omega} w_{\mid n} T_n$, henceforth for all $m \geq n_0$, $b_{\mid m} \in \bigcup_{n < n_0} w_{\mid n} T_n$. Pick $n_1 < n_0$ such that there are infinitely many $b_{\mid m}$ ($m \geq n_0$) in $w_{\mid n_1} T_{n_1}$. Therefore, there are infinitely many prefixes of $b^{\mid n_1}$ in $T_{n_1}$. As $T_{n_1}$ is prefix-closed, this yields $b^{\mid n_1} \in [T_{n_1}]$. Hence $b = b_{\mid n_1} b^{\mid n_1} = w_{\mid n_1} b^{\mid n_1} \in w_{\mid n_1} [T_{n_1}]$, which completes the proof.
    \end{enumerate}
\end{proof}


The aforementioned way to construct trees along an infinite branch is not very handy, because the union may not be disjoint as trees $T_n$ can 'overflow' along the branch $w$. It is more convenient to ask for a family $(T_u)_{u \in w^\circ}$ parametrised by the off-words of $w$, following the destructive vision.

\begin{proposition}
    If $w \in \Sigma^\omega$ and $(T_u)_{u \in w^\circ}$ is a family of trees, let
    \[ T = \Pref(w) \cup \biguplus_{u \in w^\circ} uT_u = \biguplus_{n \in \omega} w_{\mid n} \left( \{\varepsilon\} \cup \biguplus_{a \in \Sigma \setminus \{w_{n+1}\}} aT_{w_{\mid n}a}\right) \]
    Then
    \begin{enumerate}
        \item $T \in T(\Sigma)$
        \item $[T] = \{w\} \cup \biguplus_{u \in w^\circ} u[T_u]$
    \end{enumerate}
\end{proposition}

\begin{proof} Let be the non-empty trees $T_n = \{\varepsilon\} \cup \biguplus_{a \in \Sigma \setminus \{w_{n+1}\}} aT_{w_{\mid n} a}$ for every $n \in \omega$, so that $T = \bigcup_{n \in \omega} w_{\mid n} T_n$.
    \begin{enumerate}
        \item The $T_n$ are trees according to Proposition~\ref{prop:root_construct}. Therefore, using Proposition~\ref{prop:branch_construct}, $T = \bigcup_{n \in \omega} w_{\mid n} T_n$ is a tree.
        \item Proposition~\ref{prop:root_construct} yields $[T_n] = \bigcup_{a \in \Sigma \setminus \{w_{n+1}\}} a[T_{w_{\mid n}a}]$. Proposition~\ref{prop:branch_construct} yields $[T] = \left[ \bigcup_{n \in \omega} w_{\mid n} T_n \right] = \{w\} \cup \bigcup_{n \in \omega} w_{\mid n} [T_n]$. Combining the two results,
              \[ [T] = \{w\} \cup \bigcup_{n \in \omega} \bigcup_{a \in \Sigma \setminus \{w_{n+1}\}} w_{\mid n} a [T_{w_{\mid n}a}] = \{w\} \cup \biguplus_{u \in w^\circ} u[T_u] \]
    \end{enumerate}
\end{proof}

And as expected, we then have
\begin{proposition}
    If $T = \Pref(w) \cup \biguplus_{u \in w^\circ} uT_u$ and $u \in w^\circ$, then $T_u = u^{-1}T$.
\end{proposition}

\begin{proof}
    As $u^{-1}uL = L$ and $u^{-1}vL = \emptyset$ for $v \in w^\circ \setminus \{u\}$ (because $u$ and $v$ are not a prefix of each other), we have
    \[ u^{-1} T = u^{-1}\left( \Pref(w) \cup \biguplus_{u \in w^\circ} uT_u \right) = u^{-1} \Pref(w) \cup \bigcup_{v \in w^\circ} u^{-1}v T_v = T_u \]
\end{proof}

In particular, for any tree $T \in T(\Sigma)$ and $w \in [T]$ we have
\begin{align*}
     & T = \Pref(w) \uplus \biguplus_{u \in w^\circ} uu^{-1}T  \\
     & [T] = \{w\} \uplus \biguplus_{u \in w^\circ} uu^{-1}[T]
\end{align*}

\section{The Cantor-Bendixson Derivative}

This section introduces the Cantor-Bendixson derivative, an operation on topological spaces which consists in deleting isolated points. We will then see that trees themselves are topological spaces and how derivation works on them.

\subsection{Derivating Topological Spaces}

We assume familiarity with elementary topology and only give some of the basics definitions to fix notation.

A topological space is a set $X$ equipped with a set $\Omega(X) \subseteq P(X)$ such that $\emptyset, X \in \Omega(X)$ and $\Omega(X)$ is closed under arbitrary unions and finite intersections. Elements of $\Omega(X)$ are called the \emph{open sets} of $X$. Their complements $C(X)$ are called \emph{closed sets} of $X$.

Given a subset $Y \subseteq X$ of a topological space $X$, the subset topology is defined as
\[ \Omega(Y) = \{U \cap Y \mid U \in \Omega(X)\} \]

A point $x \in X$ is \emph{isolated} if $\{x\} \in \Omega(X)$.

\begin{definition}[Cantor-Bendixson derivative]
    The \emph{derivative} of a topological space $X$ is the set
    \[ X' = X \setminus \{ x \in X \mid x \text{ is isolated}\} \]
    equipped with the subset topology inherited from $X$.
\end{definition}
As $X \setminus X'$ is the union of its open singletons, the set $X'$ is closed in $X$.

\begin{lemma} \label{lem:isolated_subset}
    Let $Z$ be a topological space and $X \subseteq Y \subseteq Z$ with the subset topologies. For any $x \in X$, if $x$ is isolated in $Y$, then $x$ is isolated in $X$.
\end{lemma}
\begin{proof}
    Let $x \in X$. Assume $x$ is isolated in $Y$, then $\{x\} \in \Omega(Y)$ so there is $U \in \Omega(Z)$ such that $\{x\} = U \cap Y$. As $x \in X$, $\{x\} = \{x\} \cap X = U \cap Y \cap X = U \cap X \in \Omega(X)$ so that $x$ is isolated in $X$.
\end{proof}

As a consequence, the Cantor-Bendixson derivative is monotone.
\begin{corollary}
    Let $Z$ be a topological space and $X \subseteq Y \subseteq Z$ with the subset topologies. Then $X' \subseteq Y'$.
\end{corollary}
\begin{proof}
    Let $x \in X'$. As $X' \subseteq X$ we have $x \in X$. If $x$ is isolated in $Y$, then by Lemma~\ref{lem:isolated_subset}, $x$ is isolated in $X$ which is contradicting $x \in X'$. So $x$ is not isolated in $Y$, that is, $x \in Y'$.
\end{proof}

The Cantor-Bendixson derivative preserves the fact of being homeomorphic.
\begin{proposition}
    Let $X$ and $Y$ be homeomorphic topological spaces. Then $X'$ and $Y'$ are homeomorphic.
\end{proposition}
\begin{proof}
    Let $f : X \to Y$ be a homeomorphism. We show that it restricts to a homeomorphism $f' : X' \to Y'$.
    \begin{itemize}
        \item First see that if $f(x) \in Y$ is isolated, then $f^{-1}(\{f(x)\}) = \{x\}$ is open, hence $x \in X$ is isolated. This proves that the restriction $f' : X' \to Y'$ is well-defined.
        \item Similarly, the restriction $(f^{-1})' : Y' \to X'$ is well-defined. It is straightforward to check that $f'$ and $(f^{-1})'$ are inverse to each other, showing that $f'$ is a bijection.
        \item Let $U \cap Y' \in \Omega(Y')$, where $U \in \Omega(Y)$. Then $f'^{-1}(U \cap Y') = f^{-1}(U) \cap X' \in \Omega(X')$, whence $f'$ is continuous. \item Similarly, $(f^{-1})'$ is continuous, hence $f' : X' \to Y'$ is a homeomorphism.
    \end{itemize}
\end{proof}

\subsection{Derivating Trees}

\paragraph{Topologies on Trees}

\begin{proposition}
    The sets $\{u\Sigma^\omega \mid u \in \Sigma^*\}$ form a clopen basis of the product topology of $\Sigma^\omega$, where $\Sigma$ is given the discrete topology.
\end{proposition}

\begin{proof}
    Let $C = \{u\Sigma^\omega \mid u \in \Sigma^*\}$ be the set of \emph{cones}.
    Recall that the product topology of $\Sigma^\omega$ is generated by products $\prod_{n \in \omega} U_n$ where $U_n \in \Omega(X)$, with $U_n = \Sigma$ for all but finitely many $n$s. First we see that $C \subseteq \Omega(\Sigma^\omega)$ because $u\Sigma^\omega = \{u_1\} \times \ldots \times \{u_n\} \times \Sigma^\omega$ and every letter is open in the discrete topology of $\Sigma$. Conversely, any basic open of $\Omega(\Sigma^\omega)$ can be written
    \[ \prod_{1 \leq i \leq n} U_i \times \Sigma^\omega = \bigcup_{u \in \Sigma^n, u_i \in U_i} u\Sigma^\omega \]
    so that the topology generated by $C$ is exactly the product topology. Moreover, for any $w \in u\Sigma^\omega \cap v\Sigma^\omega$, assuming without loss of generality that $|u| \geq |v|$, we have $w \in u\Sigma^\omega \subseteq u\Sigma^\omega \cap v\Sigma^\omega$. As, additionally, $C$ covers $\Sigma^\omega$ because $\Sigma^\omega = \varepsilon\Sigma^\omega \in C$, we have that $C$ is a base of the topology its generates. Finally, let us show that every $u\Sigma^\omega$ is also closed. We have
    \[ \Sigma^\omega \setminus u\Sigma^\omega = \bigcup_{v \in \Sigma^{|u|} \setminus \{u\}} v\Sigma^\omega \in \Omega(\Sigma^\omega)\]
\end{proof}


\paragraph{Steady Convergence.}
Let $w^{(n)} \in \Sigma^\omega$ for every $n \in \omega$. If there is $w \in \Sigma^\omega$ such that $\bigcap_{n \in \omega} w^{(n)}_{\mid n} \Sigma^\omega = \{w\}$, then we write $w^{(n)} \to w$ and say that the sequence $w^{(n)}$ converges steadily to $w$.


\begin{lemma} \label{lem:sequential_convergence}
    Let $C \in C(\Sigma^\omega)$ and $w^{(n)} \to w$ with all $w^{(n)} \in C$, then $w \in C$.
\end{lemma}
\begin{proof}
    As $C$ is closed, $\Sigma^\omega \setminus C$ is open, hence we can write
    \[ \Sigma^\omega \setminus C = \bigcup_{i \in I} u_i \Sigma^\omega \]
    for some finite words $u_i \in \Sigma^*$. Assume towards a contradiction that $w \notin C$. Let $i \in I$ be such that $w \in u_i\Sigma^\omega$. Let $n = |u_i|$, then as $w \in w^{(n)}_{\mid n} \Sigma^\omega$ we have $u_i = w^{(n)}_{\mid n}$. Hence $w^{(n)} \in u_i \Sigma^\omega \subseteq \Sigma^\omega \setminus C$ which is a contradiction.
\end{proof}

\paragraph{Closed Subsets are exactly Pruned Trees.}

A \emph{pruned tree} is a tree $T \in T(\Sigma)$ such that for every $u \in T$, there is some $a \in \Sigma$ with $ua \in T$. Intuitively, a tree is pruned if it has no \emph{finite} branch. The set of pruned trees is denoted by $pT(\Sigma) \subseteq T(\Sigma)$.

\begin{proposition} \label{prop:closed_pruned_bij}
    There is a bijective correspondence between closed subsets of $\Sigma^\omega$ and the set of pruned trees on $\Sigma$:
    \[ C(\Sigma^\omega) \cong pT(\Sigma) \]
\end{proposition}

\begin{proof}
    Define the maps
    \begin{align*}
        \varphi \colon & C(\Sigma^\omega) \to pT(\Sigma)                                                    \\
                       & C \mapsto \Pref(C) = \{u \in \Sigma^* \mid C \cap u\Sigma^\infty \neq \emptyset \} \\
        \psi \colon    & pT(\Sigma) \to C(\Sigma^\omega)                                                    \\
                       & T \mapsto [T] = \{w \in \Sigma^\omega \mid \Pref(w) \subseteq T\}
    \end{align*}
    First we show that $\varphi$ and $\psi$ are well-defined. Let $C \in C(\Sigma^\omega)$. Then $\Pref(C)$ is prefix-closed hence a tree. Let us show that it is pruned. Let $w_{\mid n} \in \Pref(C)$ for some $w \in C$ and $n \in \omega$. Then, because $w \in \Sigma^\omega$, we have $w_{\mid n+1} = w_{\mid n} w_{n+1} \in \Pref(C)$, hence $\Pref(C)$ is pruned and $\varphi$ is well-defined. Now let $T \in T(\Sigma)$ be a (not necessarily pruned) tree and let us show that $[T]$ is closed in $\Sigma^\omega$. Because
    \begin{align*}
        w \in \Sigma^\omega \setminus [T]
         & \iff \exists n \in \omega, w_{\mid n} \in \Sigma^* \setminus T \\
         & \iff \exists u \in \Sigma^* \setminus T, w \in u\Sigma^\omega  \\
         & \iff w \in \bigcup_{u \in \Sigma^* \setminus T} u\Sigma^\omega
    \end{align*}
    we have
    \begin{align*}
        \Sigma^\omega \setminus [T]
         & = \bigcup_{u \in \Sigma^* \setminus T} u\Sigma^\omega \in \Omega(\Sigma^\omega)
    \end{align*}
    Now let us show that $\psi \circ \varphi$ and $\varphi \circ \psi$ are identities. Let $T \in pT(\Sigma)$ and let us show that $\Pref([T]) = T$ by double inclusion. Let $u \in \Pref([T])$, we can find $w \in [T]$ and $n \in \omega$ such that $u = w_{\mid n}$. As $w \in [T]$, we know $w_{\mid n} \in T$, hence $u \in T$. Conversely, let $u \in T$. As $T$ is pruned, we can define, by repetitively extending $u$ into arbitrarily long words in $T$, some $w \in [T]$ such that $w_{\mid n} = u$. Hence $u \in \Pref([T])$.

    Now let $C \in C(\Sigma^\omega)$ and show that $[\Pref(C)] = C$ by double inclusion. Let $w \in C$ and $n \in \omega$, then $w_{\mid n} \in \Pref(C)$ hence $w \in [\Pref(C)]$. Conversely, let $w \in [\Pref(C)]$. For any $n \in \omega$, we know $w_{\mid n} \in \Pref(C)$, hence there is $w^{(n)} \in C$ such that $w^{(n)}_{\mid n} = w_{\mid n}$. We have $\bigcap_{n \in \omega} w^{(n)}_{\mid n} \Sigma^\omega = \{w\}$ i.e. $w^{(n)} \to w$. According to Lemma~\ref{lem:sequential_convergence}, the limit $w \in C$.
\end{proof}

\begin{corollary} \label{cor:pruning}
    For any tree $T \in T(\Sigma)$, the tree $\Pref([T])$ is the unique pruned tree that has the same infinite branches as $T$.
\end{corollary}

\begin{proof}
    Let $T \in T(\Sigma)$. Let $u \in \Pref([T])$, we can write $u = w_{\mid n}$ for some $w \in [T]$ and $n \in \omega$. Then $w_{\mid n+1} \in \Pref([T])$ as well, so that $\Pref([T])$ is pruned. Now let us show that $[T] = [\Pref([T])]$. According to the proof of Proposition~\ref{prop:closed_pruned_bij}, $[T]$ is closed. Therefore as $\psi \circ \varphi$ is the identity, we have $[\Pref([T])] = [T]$. Moreover, $\Pref([T])$ is the unique pruned tree with that property. Indeed, let $T_1, T_2 \in pT(\Sigma)$ be such that $[T_1] = [T_2]$. Then, using that $\varphi \circ \psi$ is the identity, $T_1 = \Pref([T_1]) = \Pref([T_2]) = T_2$.
\end{proof}

From now on, we completely identify $C(\Sigma^\omega)$ and $pT(\Sigma)$, using explicitly or implicitly the bijection from Proposition~\ref{prop:closed_pruned_bij}. That implies identifying a pruned tree with its set of infinite branches. The map $T(\Sigma) \to pT(\Sigma)$ from Corollary~\ref{cor:pruning} given by $T \mapsto \Pref([T])$ consists in pruning a tree, which is, cutting its finite branches. It allows to sometimes identify also non-necessarily pruned trees with their set of infinite branches.

Any tree $T \in T(\Sigma)$ is given the subset topology inherited from the canonical topology on $\Sigma^\omega$.

\begin{proposition}
    Let $T \in T(\Sigma)$. For an infinite branch $w \in [T]$, the following are equivalent:
    \begin{enumerate}
        \item The point $w$ is isolated in $[T]$.
        \item There exists $N \in \omega$ such that $w_{\mid N}\Sigma^\omega \cap [T] = \{w\}$.
        \item The set $\{n \in \omega \mid \exists w' \in [T] \setminus \{w\}, w_{\mid n} = w'_{\mid n}\}$ is finite.
        \item There is no sequence $w^{(n)} \in [T] \setminus \{w\}$ such that $w^{(n)} \to w$.
    \end{enumerate}
\end{proposition}
\begin{proof} A $w \in [T]$ is isolated iff $\{w\}$ is open in the subset topology inherited from $\Sigma^\omega$, that is, iff there is $U \in \Omega(\Sigma^\omega)$ such that $U \cap [T] = \{w\}$.

    $\boxed{1 \Leftrightarrow 2}$ Let $U \in \Omega(\Sigma^\omega)$ such that $U \cap [T] = \{w\}$. Using the clopen basis, find $u \in \Sigma^*$ such that $w \in u\Sigma^\omega \subseteq U$. Let $n = |u|$, then $u = w_{\mid n}$ and by intersecting both sides of $U \cap [T] = \{w\}$ with $w_{\mid n}\Sigma^\omega$ we get $w_{\mid n} \Sigma^\omega \cap [T] = \{w\}$. Conversely, if for some $n \in \omega$ we have $w_{\mid n} \Sigma^\omega \cap [T] = \{w\}$ then $w$ is isolated in $[T]$ because $w_{\mid n} \Sigma^\omega \in \Omega(\Sigma^\omega)$.

    $\boxed{2 \Leftrightarrow 3}$ Let $N \in \omega$ such that $w_{\mid N} \Sigma^\omega \cap [T] = \{w\}$. For any $n \geq N$, if $w' \in [T]$ and $w_{\mid n} = w'_{\mid n}$ then $w' \in w_{\mid n} \Sigma^\omega \cap [T] \subseteq w_{\mid N} \Sigma^\omega \cap [T] = \{w\}$ hence $w' = w$. Therefore, $B := \{n \in \omega \mid \exists w' \in [T] \setminus \{w\}, w_{\mid n} = w'_{\mid n}\} \subseteq \{0,\ldots,N-1\}$ is finite. Conversely, if $B$ is finite then letting $N = \max B + 1 \in \omega$, we have $w_{\mid n} \cap [T] = \{w\}$.

    $\boxed{2 \Leftrightarrow 4}$ If $w^{(n)} \in [T] \setminus \{w\}$ is a sequence such that $w^{(n)} \to w$, then for all $n \in \omega$ we have $w_{\mid n} \Sigma^\omega \cap [T] \supseteq \{w,w^{(n)}\}$ so this set can never be a singleton. Conversely, if for every $n \in \omega$ there is $w^{(n)} \in [T] \setminus \{w\}$ such that $w^{(n)}_{\mid n} = w_{\mid n}$ then $\bigcap_{n \in \omega} w^{(n)}_{\mid n} \Sigma^\omega = \{w\}$ i.e. $w^{(n)} \to w$.
\end{proof}

\begin{lemma} \label{lem:convergence_corestrict}
    Let $w^{(n)} \to w$.
    \begin{enumerate}
        \item For any $p\in\omega$, $w^{(p+n)\mid p} \to w^{\mid p}$. Moreover, for every $n \in \omega$, $w^{(p+n)\mid p} = w^{\mid p} \iff w^{(p+n)} = w$.
        \item For any $u \in \Sigma^*$, $uw^{(n)} \to uw$.
    \end{enumerate}
\end{lemma}

\begin{proof}
    \begin{enumerate}

        \item Assume $w^{(n)} \to w$ i.e. $\bigcap_{n \in \omega} w^{(n)}_{\mid n} \Sigma^\omega = \{w\}$. That is, for every $n \in \omega$ we have $w^{(n)}_{\mid n} = w_{\mid n}$. We must prove that for every $n \in \omega$ we have $(w^{(p+n)\mid p})_{\mid n} = (w^{\mid p})_{\mid n}$. This follows from
              \[ w^{(p+n)}_{\mid p} (w^{(p+n)\mid p})_{\mid n} = w^{(p+n)}_{\mid p+n} = w_{\mid p+n} = w_{\mid p} (w^{\mid p})_{\mid n} \]
              Moreover, if for some $n \in \omega$ we have $w^{(p+n)\mid p} = w^{\mid p}$, then as $w^{(p+n)}_{\mid p+n} = w_{\mid p+n}$ we have $w^{(p+n)}_{\mid p} = w_{\mid p}$ so that
              \[ w^{(p+n)} = w^{(p+n)}_{\mid p} w^{(p+n)\mid p} = w_{\mid p} w^{\mid p} = w \]
              and the converse implication is obvious.

        \item Let $p = |u|$. For every $n \leq p$ we have $(uw^{(n)})_{\mid n} = u_{\mid n} = (uw)_{\mid n}$. For $n > p$ we have $(uw^{(n)})_{\mid n} = uw^{(n)}_{\mid n-p} = uw_{\mid n-p} = (uw^{(n)})_{\mid n}$.

    \end{enumerate}
\end{proof}

\begin{proposition}
    For any $T \in T(\Sigma)$ and $u \in \Sigma^*$,
    \begin{enumerate}
        \item The set $u[T]$ is closed and $(u[T])' = u[T]'$.
        \item The set $u^{-1}[T]$ is closed and $(u^{-1}[T])' = u^{-1}[T]'$.
    \end{enumerate}
\end{proposition}

\begin{proof}
    Let $T \in T(\Sigma)$ and $u \in \Sigma^*$.
    \begin{enumerate}
        \item As $u[T] = [\Pref(u) \cup uT]$, $u[T]$ is closed.
        \item[] $\boxed{\subseteq}$ Let $w \in (u[T])'$. As $w$ is not isolated in $u[T]$, we can find $w^{(n)} \in u[T] \setminus \{w\}$ such that $w^{(n)} \to w$. Let $p = |u|$. By Lemma~\ref{lem:convergence_corestrict}, we have $w^{(p+n)\mid p} \to w^{\mid p}$ and $w^{(p+n)\mid p} \in [T] \setminus \{w^{\mid p}\}$. Therefore $w^{\mid p} \in [T]'$ so that $w \in u[T]'$.
        \item[] $\boxed{\supseteq}$ Let $w \in u[T]'$ and $p = |u|$. Then $w^{\mid p} \in [T]'$ so that there is a sequence $z^{(n)} \to w^{\mid p}$ with $z^{(n)} \in [T] \setminus \{w^{\mid p}\}$. According to Lemma~\ref{lem:convergence_corestrict}, we have $uz^{(n)} \to uw^{\mid p} = w_{\mid p}w^{\mid p} = w$, and obviously $uz^{(n)} \in u[T] \setminus \{w\}$, therefore $w \in (u[T])'$.
        \item As $u^{-1}[T] = [u^{-1}T]$, $u^{-1}[T]$ is closed.
        \item[] $\boxed{\subseteq}$ Let $w \in (u^{-1}[T])'$. Let be a sequence $w^{(n)} \to w$ with $w^{(n)} \in u^{-1}[T]$, i.e. $uw^{(n)} \in [T]$, and $w^{(n)} \neq w$ for all $n \in \omega$. We have $uw^{(n)} \in [T] \setminus \{uw\}$. Also, by Lemma~\ref{lem:convergence_corestrict}, we have $uw^{(n)} \to uw$. Hence $uw \in [T]'$ so $w \in u^{-1}[T]'$.
        \item[] $\boxed{\supseteq}$ Let $w \in u^{-1}[T]'$, that is, $uw \in [T]'$. Find a sequence $v^{(n)} \to uw$ with $v^{(n)} \in [T] \setminus \{uw\}$ for all $n \in \omega$. Let $p = |u|$. As $v^{(n)} \to uw$ we have $v^{(p+n)}_{\mid p} = u$ for every $n \in \omega$, hence $uv^{(p+n)\mid p} = v^{(p+n)}_{\mid p} v^{(p+n)\mid p} = v \in [T]$, that is, $v^{(p+n)\mid p} \in u^{-1}[T]$. By Lemma~\ref{lem:convergence_corestrict}, we have $v^{(p+n)\mid p} \to (uw)^{\mid p} = w$. Still by Lemma~\ref{lem:convergence_corestrict}, as $v^{(n)} \neq uw$ we have $v^{(p+n)\mid p} \neq w$. Therefore $w \in (u^{-1}[T])'$.
    \end{enumerate}
\end{proof}

\section{The Cantor-Bendixson Rank}

The Cantor-Bendixson derivative can be transfinitely iterated. The derivating process always stabilises at some ordinal $\alpha$. This $\alpha$ is called the Cantor-Bendixson rank of the topological space. This section studies the interplay between rank and tree decompositions.

\subsection{On Topological Spaces}

One can transfinitely iterate the topological derivative as follows. For a topological space $X$, let
\begin{align*}
     & X^{(0)} =  X                                          &                               \\
     & X^{(\alpha+1)} = (X^{(\alpha)})'                      &                               \\
     & X^{(\lambda)} = \bigcap_{\beta < \lambda} X^{(\beta)} & \lambda \text{ limit ordinal}
\end{align*}

The sequence $X^{(\alpha)}$ is decreasing, hence stationary.

\begin{definition}[Cantor-Bendixson rank]
    For any topological space $X$, there is a least ordinal $\alpha$ such that $X^{(\alpha)} = X^{(\alpha+1)}$. This ordinal is called the \emph{Cantor-Bendixson rank} of $X$ and is denoted as $\alpha = \CB(X)$.
\end{definition}

By definition, the space $X^{(\CB(X))}$ is \emph{perfect} in the sense that it has no isolated points. It is called the \emph{perfect kernel} of $X$.

\subsection{On Trees}

\begin{definition}[Thin tree]
    A tree $T \in T(\Sigma)$ is \emph{thin} if its perfect kernel is empty, that is, $[T]^{(\CB(T))} = \emptyset$.
\end{definition}

\begin{proposition}
    If $\Sigma$ is countable, thin trees in $T(\Sigma)$ are exactly the trees with countably many infinite branches.
\end{proposition}
\begin{proof}
    The proof relies on some basic descriptive-set-theoretic notions that have not been introduced above. The reader can use~\cite{Kechris95} as a general reference. Assume $\Sigma$ is countable. Then $\Sigma^\omega$ is Polish. Moreover, as any closed subset of a Polish space is polish, any tree in $T(\Sigma)$ is Polish. Any Polish space $X$ can be written as a disjoint union $X = P \uplus U$ where $P$ is its maximal perfect Polish subspace (the perfect kernel) and $U$ is the union of all its countable open subsets (itself countable). Non-empty perfect sets in a Polish space have the cardinality of the continuum.
    \begin{enumerate}
        \item[$\Rightarrow$] Let $T \in T(\Sigma)$ thin. Then the perfect kernel $P$ of $[T]$ is empty, so its countable complement is $U = [T]$. Therefore, $T$ has countably many infinite branches.
        \item[$\Leftarrow$] Let $T \in T(\Sigma)$ have countably many infinite branches. The perfect kernel $P$ of $[T]$ cannot have the cardinality of the continuum, therefore it is empty, so that $T$ is thin.
    \end{enumerate}
\end{proof}

\begin{remark}
    At the contrary, when there is no restriction on the cardinality of $\Sigma$, thin trees can have many infinite branches. For instance, the tree $\Pref(\{s^\omega \mid s \in \Sigma\})$ is thin, but has $|\Sigma|$-many infinite branches.
\end{remark}


It is also interesting to have a notion of derivative for trees seen as sets of finite words themselves.

\begin{definition}
    For any $T \in T(\Sigma)$ and ordinal $\alpha$, let $T^{(\alpha)} = \Pref([T]^{(\alpha)})$.
\end{definition}

\begin{remark} \label{rem:same_branches}
    Then we have $[T^{(\alpha)}] = [T]^{(\alpha)}$.
\end{remark}





We obtain a notion of Cantor-Bendixson rank which is equivalent to the other one.
\begin{proposition}
    Let $\CB(T)$ be the minimal ordinal $\alpha$ such that $T^{(\alpha)} = T^{(\alpha+1)}$. Then $\CB(T) = \CB([T])$.
\end{proposition}
\begin{proof}
    It is sufficient to prove $T^{(\alpha)} = T^{(\alpha+1)} \iff [T]^{(\alpha)} = [T]^{(\alpha+1)}$ for every ordinal $\alpha$. Left-to-right implication is obvious considered Remark~\ref{rem:same_branches}. Assuming $[T]^{(\alpha)} = [T]^{(\alpha+1)}$, we have \[ T^{(\alpha)} = \Pref([T^{(\alpha)}]) = \Pref([T]^{(\alpha)}) = \Pref([T]^{(\alpha+1)}) = \Pref([T^{(\alpha+1)}]) = T^{(\alpha+1)} \] because by definition $T^{(\alpha)}$ is pruned.
\end{proof}
In particular, any tree with no infinite branch, and \emph{a fortiori} any finite tree has rank $0$.

For every $w \in [T]'$, we have
\[ [T]' = \{w\} \cup \biguplus_{u \in w^\circ} uu^{-1}[T]' \]
therefore, as $u^{-1}[T]' = [u^{-1}T]'$
\[ [T]' = \{w\} \cup \biguplus_{u \in w^\circ} u[u^{-1}T]' \]
Hence derivating a tree along a non-isolated branch is the same as derivating all the trees plugged along that branch.


\begin{proposition}
    For any $n \in \omega$, $[B_{n+1}]' = [B_n]$.
\end{proposition}
\begin{proof}
    Recall that $B_0 = \emptyset$ and $B_{n+1} = \Pref(a^\omega) \cup \biguplus_{k\in \omega} a^k b B_n^\perp$. Prove the result by induction. Initialisation: $[B_1]' = \{a^\omega\}' = \emptyset = [\emptyset] = [B_0]$. Now, assuming $[B_{n+1}]' = [B_n]$, we have, as $a^\omega \in [B_{n+2}]'$,
    \[ [B_{n+2}]' = \{a^\omega\} \cup \biguplus_{k \in \omega} a^k b [B_{n+1}^\perp]' = \{a^\omega\} \cup \biguplus_{k \in \omega} a^k b [B_n]^{\perp} = [B_{n+1}] \]
    where the middle equality is using the induction hypothesis along with commutation of $\perp$ with $'$ and $[-]$:
    \[ [B_{n+1}^\perp]' = ([B_{n+1}]^\perp)' = ([B_{n+1}]')^\perp = [B_n]^\perp \]
\end{proof}

We can iterate transfinitely that construction.
\begin{proposition}
    For every $T \in T(\Sigma)$, every ordinal $\alpha$ and every $w \in [T]^{(\alpha)}$,
    \[ [T]^{(\alpha)} = \{w\} \cup \biguplus_{u \in w^\circ} u[u^{-1}T]^{(\alpha)} \]
\end{proposition}

The Cantor-Bendixson rank of a tree is denoted by $\CB(T)$.

\begin{proposition}
    If $\Sigma$ is finite, then for every thin $T \in T(\Sigma)$, $\CB(T)$ is either zero or a successor ordinal. This is not true in general for infinite $\Sigma$.
\end{proposition}
\begin{proof}
    The first part is folklore. Let us give a counterexample when $\Sigma$ is infinite. Taking $\Sigma = \omega \cup \{a,b\}$, let $T = \{\varepsilon\} \cup \biguplus_{n \in \omega} nB_n$. Then $[T] = \biguplus_{n \in \omega} n[B_n]$. As $[B_0] = \emptyset$ and $[B_{n+1}]' = [B_n]$, it is easy to see that for any $k \in \omega$, $[B_n]^{(k)} = [B_{n-k}] \neq \emptyset$ if $k < n$,  $[B_n]^{(k)} = \emptyset$ otherwise. Therefore, for any $k \in \omega$, $[T]^{(k)} = \biguplus_{n > k} n[B_{n-k}] \neq \emptyset$. But
    \[ [T]^{(\omega)} = \bigcap_{k \in \omega} \biguplus_{n > k} n[B_{n-k}] = \emptyset \]
    so that $\CB(T) = \omega$ is not a successor ordinal.
\end{proof}

\end{document}

%% file: trees.tikzstyles
\tikzstyle{tree node}=[fill=black, draw=black, shape=circle, minimum size=1mm, inner sep=0.5 pt]
\tikzstyle{ghost}=[fill=none, draw=none, shape=circle, tikzit fill=white, tikzit draw={rgb,255: red,191; green,191; blue,191}]

\tikzstyle{tree edge}=[-]
\tikzstyle{to infinity}=[-, densely dotted, tikzit draw={rgb,255: red,191; green,191; blue,191}]
\tikzstyle{main branch}=[-, very thick, tikzit draw={rgb,255: red,231; green,0; blue,3}, draw=red]
\tikzstyle{new edge style 0}=[-, draw=red, tikzit draw={rgb,255: red,255; green,191; blue,191}, densely dashed, very thick]

%% file: hat_tree.tikz
\begin{tikzpicture}
	\begin{pgfonlayer}{nodelayer}
		\node [style=tree node] (0) at (0, 0) {};
		\node [style=tree node] (1) at (-1, -1) {};
		\node [style=tree node] (2) at (1, -1) {};
		\node [style=tree node] (3) at (-2, -2) {};
		\node [style=tree node] (4) at (2, -2) {};
		\node [style=tree node] (5) at (-3, -3) {};
		\node [style=tree node] (6) at (3, -3) {};
		\node [style=ghost] (7) at (-4, -4) {};
		\node [style=ghost] (8) at (4, -4) {};
	\end{pgfonlayer}
	\begin{pgfonlayer}{edgelayer}
		\draw [style=tree edge] (0) to (1);
		\draw [style=tree edge] (1) to (3);
		\draw [style=tree edge] (3) to (5);
		\draw [style=tree edge] (0) to (2);
		\draw [style=tree edge] (4) to (6);
		\draw [style=to infinity] (5) to (7);
		\draw [style=to infinity] (6) to (8);
		\draw [style=tree edge] (2) to (4);
	\end{pgfonlayer}
\end{tikzpicture}

%% file: branching1_tree.tikz
\begin{tikzpicture}
	\begin{pgfonlayer}{nodelayer}
		\node [style=tree node] (0) at (11, 0) {};
		\node [style=tree node] (1) at (10, -1) {};
		\node [style=tree node] (2) at (12, -1) {};
		\node [style=tree node] (3) at (9, -2) {};
		\node [style=tree node] (4) at (13, -2) {};
		\node [style=tree node] (5) at (8, -3) {};
		\node [style=tree node] (6) at (14, -3) {};
		\node [style=ghost] (7) at (6.75, -4.25) {};
		\node [style=ghost] (8) at (15.25, -4.25) {};
		\node [style=tree node] (9) at (11, -2) {};
		\node [style=tree node] (10) at (12, -3) {};
		\node [style=ghost] (11) at (13.25, -4.25) {};
		\node [style=tree node] (13) at (10, -3) {};
		\node [style=ghost] (14) at (11.25, -4.25) {};
		\node [style=ghost] (17) at (9.25, -4.25) {};
		\node [style=none] (18) at (8, -1) {$B_2$};
		\node [style=tree node] (65) at (0, 0) {};
		\node [style=tree node] (66) at (-1, -1) {};
		\node [style=tree node] (68) at (-2, -2) {};
		\node [style=tree node] (70) at (-3, -3) {};
		\node [style=ghost] (72) at (-4.25, -4.25) {};
		\node [style=none] (80) at (-3, -1) {$B_1$};
	\end{pgfonlayer}
	\begin{pgfonlayer}{edgelayer}
		\draw [style=tree edge] (0) to (1);
		\draw [style=tree edge] (1) to (3);
		\draw [style=tree edge] (3) to (5);
		\draw [style=tree edge] (0) to (2);
		\draw [style=tree edge] (4) to (6);
		\draw [style=to infinity] (5) to (7);
		\draw [style=to infinity] (6) to (8);
		\draw [style=tree edge] (2) to (4);
		\draw [style=to infinity] (10) to (11);
		\draw [style=tree edge] (9) to (10);
		\draw [style=to infinity] (13) to (14);
		\draw [style=tree edge] (1) to (9);
		\draw [style=tree edge] (3) to (13);
		\draw [style=to infinity] (5) to (17);
		\draw [style=tree edge] (65) to (66);
		\draw [style=tree edge] (66) to (68);
		\draw [style=tree edge] (68) to (70);
		\draw [style=to infinity] (70) to (72);
	\end{pgfonlayer}
\end{tikzpicture}

%% file: branching2_tree.tikz
\begin{tikzpicture}
	\begin{pgfonlayer}{nodelayer}
		\node [style=tree node] (0) at (0, 0) {};
		\node [style=tree node] (1) at (1, -1) {};
		\node [style=tree node] (2) at (2, -2) {};
		\node [style=tree node] (3) at (3, -3) {};
		\node [style=tree node] (4) at (0.5, -1.5) {};
		\node [style=tree node] (5) at (1.5, -2.5) {};
		\node [style=tree node] (6) at (2.5, -3.5) {};
		\node [style=ghost] (7) at (-0.25, -2.25) {};
		\node [style=ghost] (8) at (0.75, -3.25) {};
		\node [style=ghost] (9) at (1.75, -4.25) {};
		\node [style=tree node] (10) at (-1.5, -1.5) {};
		\node [style=tree node] (11) at (-0.5, -2.5) {};
		\node [style=tree node] (12) at (-1, -3) {};
		\node [style=ghost] (13) at (-1.75, -3.75) {};
		\node [style=tree node] (14) at (0.5, -3.5) {};
		\node [style=tree node] (15) at (-3, -3) {};
		\node [style=ghost] (16) at (-4.25, -4.25) {};
		\node [style=tree node] (17) at (0.5, -0.5) {};
		\node [style=tree node] (18) at (1.5, -1.5) {};
		\node [style=tree node] (19) at (2.5, -2.5) {};
		\node [style=tree node] (20) at (0, -1) {};
		\node [style=tree node] (21) at (1, -2) {};
		\node [style=tree node] (22) at (2, -3) {};
		\node [style=ghost] (23) at (-0.75, -1.75) {};
		\node [style=ghost] (24) at (0.25, -2.75) {};
		\node [style=ghost] (25) at (1.25, -3.75) {};
		\node [style=tree node] (26) at (-1, -2) {};
		\node [style=tree node] (27) at (0, -3) {};
		\node [style=tree node] (28) at (-1.5, -2.5) {};
		\node [style=tree node] (29) at (-0.5, -3.5) {};
		\node [style=ghost] (30) at (-2.25, -3.25) {};
		\node [style=ghost] (31) at (-1.25, -4.25) {};
		\node [style=ghost] (32) at (1.25, -4.25) {};
		\node [style=tree node] (33) at (-2.5, -3.5) {};
		\node [style=ghost] (34) at (-1.75, -4.25) {};
		\node [style=ghost] (35) at (-3.25, -4.25) {};
		\node [style=none] (36) at (-3, -1) {$B_3$};
		\node [style=tree node] (37) at (3.5, -3.5) {};
		\node [style=ghost] (38) at (-0.25, -4.25) {};
		\node [style=ghost] (39) at (2.75, -4.25) {};
		\node [style=ghost] (40) at (4.25, -4.25) {};
		\node [style=tree node] (41) at (11, 0) {};
		\node [style=tree node] (42) at (11.75, -0.75) {};
		\node [style=tree node] (43) at (13.25, -2.25) {};
		\node [style=tree node] (44) at (14, -3) {};
		\node [style=tree node] (45) at (11.25, -1.25) {};
		\node [style=tree node] (46) at (12.75, -2.75) {};
		\node [style=tree node] (47) at (13.5, -3.5) {};
		\node [style=ghost] (48) at (10.25, -2.25) {};
		\node [style=ghost] (49) at (11.75, -3.75) {};
		\node [style=ghost] (50) at (12.75, -4.25) {};
		\node [style=tree node] (51) at (9.5, -1.5) {};
		\node [style=tree node] (52) at (10.25, -2.25) {};
		\node [style=tree node] (53) at (9.75, -2.75) {};
		\node [style=ghost] (54) at (8.75, -3.75) {};
		\node [style=tree node] (55) at (11.75, -3.75) {};
		\node [style=tree node] (56) at (8, -3) {};
		\node [style=ghost] (57) at (6.75, -4.25) {};
		\node [style=tree node] (58) at (12.5, -1.5) {};
		\node [style=tree node] (59) at (12, -2) {};
		\node [style=ghost] (60) at (11, -3) {};
		\node [style=tree node] (61) at (11, -3) {};
		\node [style=tree node] (62) at (10.5, -3.5) {};
		\node [style=ghost] (63) at (9.75, -4.25) {};
		\node [style=ghost] (64) at (12.25, -4.25) {};
		\node [style=tree node] (65) at (8.75, -3.75) {};
		\node [style=ghost] (66) at (9.25, -4.25) {};
		\node [style=ghost] (67) at (8.25, -4.25) {};
		\node [style=none] (68) at (8, -1) {$B_4$};
		\node [style=tree node] (69) at (14.75, -3.75) {};
		\node [style=ghost] (70) at (11.25, -4.25) {};
		\node [style=ghost] (71) at (14.25, -4.25) {};
		\node [style=ghost] (72) at (15.25, -4.25) {};
		\node [style=tree node] (73) at (10.75, -1.75) {};
		\node [style=tree node] (74) at (11.5, -2.5) {};
		\node [style=tree node] (75) at (12.25, -3.25) {};
		\node [style=tree node] (76) at (9.25, -3.25) {};
		\node [style=tree node] (77) at (11.5, -1.5) {};
		\node [style=ghost] (78) at (12, -2) {};
		\node [style=tree node] (79) at (11, -2) {};
		\node [style=tree node] (80) at (12.25, -2.25) {};
		\node [style=tree node] (81) at (11.75, -2.75) {};
		\node [style=tree node] (82) at (12.5, -3.5) {};
		\node [style=tree node] (83) at (13, -3) {};
		\node [style=tree node] (84) at (13.75, -3.75) {};
		\node [style=ghost] (85) at (11.5, -2.5) {};
		\node [style=ghost] (86) at (12.75, -2.75) {};
		\node [style=ghost] (87) at (12.25, -3.25) {};
		\node [style=ghost] (88) at (13.5, -3.5) {};
		\node [style=ghost] (89) at (13, -4) {};
		\node [style=ghost] (90) at (14.25, -4.25) {};
		\node [style=tree node] (91) at (9.5, -3.5) {};
		\node [style=tree node] (92) at (10, -3) {};
		\node [style=tree node] (93) at (10.75, -3.75) {};
		\node [style=ghost] (94) at (10, -4) {};
		\node [style=ghost] (95) at (10.5, -3.5) {};
		\node [style=ghost] (96) at (11.25, -4.25) {};
		\node [style=tree node] (97) at (11, -1.5) {};
		\node [style=tree node] (98) at (11.5, -1) {};
		\node [style=tree node] (99) at (11.25, -1.75) {};
		\node [style=tree node] (100) at (11.75, -1.25) {};
		\node [style=tree node] (101) at (12, -2.5) {};
		\node [style=tree node] (102) at (12.5, -2) {};
		\node [style=tree node] (103) at (12.25, -1.75) {};
		\node [style=tree node] (104) at (11.75, -2.25) {};
		\node [style=tree node] (105) at (13, -2.5) {};
		\node [style=tree node] (106) at (12.5, -3) {};
		\node [style=tree node] (107) at (12.75, -3.25) {};
		\node [style=tree node] (108) at (13.25, -2.75) {};
		\node [style=tree node] (109) at (13.75, -3.25) {};
		\node [style=tree node] (110) at (14, -3.5) {};
		\node [style=ghost] (111) at (11.75, -2.25) {};
		\node [style=ghost] (112) at (12.25, -1.75) {};
		\node [style=ghost] (113) at (12.5, -3) {};
		\node [style=ghost] (114) at (13, -2.5) {};
		\node [style=ghost] (115) at (14.5, -4) {};
		\node [style=ghost] (116) at (13.75, -3.25) {};
		\node [style=ghost] (117) at (13.25, -3.75) {};
		\node [style=tree node] (118) at (9.5, -3) {};
		\node [style=tree node] (119) at (10, -2.5) {};
		\node [style=tree node] (120) at (10.25, -2.75) {};
		\node [style=tree node] (121) at (9.75, -3.25) {};
		\node [style=tree node] (122) at (11, -3.5) {};
		\node [style=tree node] (123) at (10.75, -3.25) {};
		\node [style=ghost] (124) at (10.25, -3.75) {};
		\node [style=ghost] (125) at (10.75, -3.25) {};
		\node [style=ghost] (126) at (11.5, -4) {};
	\end{pgfonlayer}
	\begin{pgfonlayer}{edgelayer}
		\draw [style=to infinity] (4) to (7);
		\draw [style=to infinity] (5) to (8);
		\draw [style=to infinity] (6) to (9);
		\draw [style=to infinity] (12) to (13);
		\draw [style=tree edge] (0) to (10);
		\draw [style=tree edge] (0) to (1);
		\draw [style=tree edge] (1) to (4);
		\draw [style=tree edge] (1) to (2);
		\draw [style=tree edge] (2) to (5);
		\draw [style=tree edge] (2) to (3);
		\draw [style=tree edge] (3) to (6);
		\draw [style=tree edge] (10) to (11);
		\draw [style=tree edge] (11) to (12);
		\draw [style=tree edge] (11) to (14);
		\draw [style=tree edge] (10) to (15);
		\draw [style=to infinity] (15) to (16);
		\draw [style=tree edge] (17) to (20);
		\draw [style=tree edge] (18) to (21);
		\draw [style=tree edge] (19) to (22);
		\draw [style=to infinity] (20) to (23);
		\draw [style=to infinity] (21) to (24);
		\draw [style=to infinity] (22) to (25);
		\draw [style=to infinity] (29) to (31);
		\draw [style=to infinity] (28) to (30);
		\draw [style=tree edge] (27) to (29);
		\draw [style=tree edge] (26) to (28);
		\draw [style=to infinity] (14) to (32);
		\draw [style=to infinity] (33) to (34);
		\draw [style=tree edge] (15) to (33);
		\draw [style=to infinity] (33) to (35);
		\draw [style=to infinity] (14) to (38);
		\draw [style=to infinity] (37) to (39);
		\draw [style=to infinity] (37) to (40);
		\draw [style=tree edge] (3) to (37);
		\draw [style=to infinity] (45) to (48);
		\draw [style=to infinity] (46) to (49);
		\draw [style=to infinity] (47) to (50);
		\draw [style=to infinity] (53) to (54);
		\draw [style=tree edge] (41) to (51);
		\draw [style=tree edge] (41) to (42);
		\draw [style=tree edge] (42) to (45);
		\draw [style=tree edge] (42) to (43);
		\draw [style=tree edge] (43) to (46);
		\draw [style=tree edge] (43) to (44);
		\draw [style=tree edge] (44) to (47);
		\draw [style=tree edge] (51) to (52);
		\draw [style=tree edge] (52) to (53);
		\draw [style=tree edge] (52) to (55);
		\draw [style=tree edge] (51) to (56);
		\draw [style=to infinity] (56) to (57);
		\draw [style=tree edge] (58) to (59);
		\draw [style=to infinity] (59) to (60);
		\draw [style=to infinity] (62) to (63);
		\draw [style=tree edge] (61) to (62);
		\draw [style=to infinity] (55) to (64);
		\draw [style=to infinity] (65) to (66);
		\draw [style=tree edge] (56) to (65);
		\draw [style=to infinity] (65) to (67);
		\draw [style=to infinity] (55) to (70);
		\draw [style=to infinity] (69) to (71);
		\draw [style=to infinity] (69) to (72);
		\draw [style=tree edge] (44) to (69);
		\draw [style=tree edge] (45) to (73);
		\draw [style=tree edge] (59) to (74);
		\draw [style=tree edge] (46) to (75);
		\draw [style=tree edge] (53) to (76);
		\draw [style=tree edge] (45) to (77);
		\draw [style=to infinity] (77) to (78);
		\draw [style=tree edge] (74) to (81);
		\draw [style=tree edge] (73) to (79);
		\draw [style=tree edge] (46) to (83);
		\draw [style=tree edge] (75) to (82);
		\draw [style=tree edge] (47) to (84);
		\draw [style=to infinity] (79) to (85);
		\draw [style=to infinity] (80) to (86);
		\draw [style=to infinity] (81) to (87);
		\draw [style=tree edge] (59) to (80);
		\draw [style=to infinity] (83) to (88);
		\draw [style=to infinity] (82) to (89);
		\draw [style=to infinity] (84) to (90);
		\draw [style=tree edge] (53) to (92);
		\draw [style=tree edge] (76) to (91);
		\draw [style=tree edge] (62) to (93);
		\draw [style=to infinity] (91) to (94);
		\draw [style=to infinity] (92) to (95);
		\draw [style=to infinity] (93) to (96);
		\draw [style=tree edge] (97) to (99);
		\draw [style=tree edge] (98) to (100);
		\draw [style=tree edge] (103) to (102);
		\draw [style=tree edge] (104) to (101);
		\draw [style=tree edge] (106) to (107);
		\draw [style=tree edge] (105) to (108);
		\draw [style=tree edge] (109) to (110);
		\draw [style=to infinity] (99) to (111);
		\draw [style=to infinity] (100) to (112);
		\draw [style=to infinity] (101) to (113);
		\draw [style=to infinity] (102) to (114);
		\draw [style=to infinity] (110) to (115);
		\draw [style=to infinity] (107) to (117);
		\draw [style=to infinity] (108) to (116);
		\draw [style=tree edge] (118) to (121);
		\draw [style=tree edge] (119) to (120);
		\draw [style=tree edge] (123) to (122);
		\draw [style=to infinity] (121) to (124);
		\draw [style=to infinity] (120) to (125);
		\draw [style=to infinity] (122) to (126);
	\end{pgfonlayer}
\end{tikzpicture}

%% file: growing_tree.tikz
\begin{tikzpicture}
	\begin{pgfonlayer}{nodelayer}
		\node [style=tree node] (97) at (0, 0) {};
		\node [style=tree node] (98) at (-1, -1) {};
		\node [style=none] (99) at (1, -1) {};
		\node [style=tree node] (100) at (-2, -2) {};
		\node [style=tree node] (102) at (-3, -3) {};
		\node [style=ghost] (104) at (-4.25, -4.25) {};
		\node [style=none] (106) at (0, -2) {};
		\node [style=none] (109) at (-1, -3) {};
		\node [style=ghost] (111) at (-1.75, -4.25) {};
		\node [style=none] (112) at (-3, -1) {$G$};
		\node [style=none] (113) at (1.5, -1.5) {$B_1^\perp$};
		\node [style=none] (114) at (0.5, -2.5) {$B_2^\perp$};
		\node [style=none] (115) at (-0.5, -3.5) {$B_3^\perp$};
		\node [style=none] (116) at (3, -2) {$=$};
		\node [style=tree node] (117) at (8.25, 0) {};
		\node [style=tree node] (118) at (7.25, -1) {};
		\node [style=tree node] (119) at (9.25, -1) {};
		\node [style=tree node] (120) at (6.25, -2) {};
		\node [style=tree node] (121) at (10.25, -2) {};
		\node [style=tree node] (122) at (5.25, -3) {};
		\node [style=tree node] (123) at (11.25, -3) {};
		\node [style=ghost] (124) at (4, -4.25) {};
		\node [style=ghost] (125) at (12.5, -4.25) {};
		\node [style=tree node] (126) at (8.25, -2) {};
		\node [style=tree node] (127) at (9.25, -3) {};
		\node [style=ghost] (128) at (10.5, -4.25) {};
		\node [style=tree node] (129) at (7.25, -3) {};
		\node [style=ghost] (130) at (8.5, -4.25) {};
		\node [style=ghost] (131) at (6.5, -4.25) {};
		\node [style=tree node] (132) at (7.75, -2.5) {};
		\node [style=tree node] (133) at (8.75, -3.5) {};
		\node [style=ghost] (134) at (7.25, -3) {};
		\node [style=ghost] (135) at (8.25, -4) {};
		\node [style=tree node] (136) at (6.75, -3.5) {};
		\node [style=tree node] (137) at (7, -3.75) {};
		\node [style=ghost] (138) at (6.25, -4) {};
		\node [style=ghost] (139) at (7.5, -4.25) {};
		\node [style=tree node] (140) at (8, -2.25) {};
		\node [style=tree node] (141) at (9, -3.25) {};
		\node [style=tree node] (142) at (7, -3.25) {};
		\node [style=tree node] (143) at (7.25, -3.5) {};
		\node [style=ghost] (144) at (8, -4.25) {};
	\end{pgfonlayer}
	\begin{pgfonlayer}{edgelayer}
		\draw [style=tree edge] (97) to (98);
		\draw [style=tree edge] (98) to (100);
		\draw [style=tree edge] (100) to (102);
		\draw [style=tree edge] (97) to (99.center);
		\draw [style=to infinity] (102) to (104);
		\draw [style=tree edge] (98) to (106.center);
		\draw [style=tree edge] (100) to (109.center);
		\draw [style=to infinity] (102) to (111);
		\draw [style=tree edge] (117) to (118);
		\draw [style=tree edge] (118) to (120);
		\draw [style=tree edge] (120) to (122);
		\draw [style=tree edge] (117) to (119);
		\draw [style=tree edge] (121) to (123);
		\draw [style=to infinity] (122) to (124);
		\draw [style=to infinity] (123) to (125);
		\draw [style=tree edge] (119) to (121);
		\draw [style=to infinity] (127) to (128);
		\draw [style=tree edge] (126) to (127);
		\draw [style=to infinity] (129) to (130);
		\draw [style=tree edge] (118) to (126);
		\draw [style=tree edge] (120) to (129);
		\draw [style=to infinity] (122) to (131);
		\draw [style=tree edge] (126) to (132);
		\draw [style=tree edge] (127) to (133);
		\draw [style=to infinity] (132) to (134);
		\draw [style=to infinity] (133) to (135);
		\draw [style=to infinity] (136) to (138);
		\draw [style=to infinity] (137) to (139);
		\draw [style=tree edge] (129) to (136);
		\draw [style=tree edge] (136) to (137);
		\draw [style=tree edge] (142) to (143);
		\draw [style=to infinity] (143) to (144);
	\end{pgfonlayer}
\end{tikzpicture}